\documentclass{icmart}
\usepackage{amsmath}

\contact[krivelev@post.tau.ac.il]{School of Mathematical Sciences, Tel Aviv University, Tel Aviv 6997801, Israel.}

\newtheorem{theorem}{Theorem}[section]
\newtheorem{corollary}[theorem]{Corollary}

\newtheorem{proposition}[theorem]{Proposition}

\newtheorem{problem}[theorem]{Problem}

\theoremstyle{definition}
\newtheorem{definition}[theorem]{Definition}

\DeclareMathOperator*{\argmax}{arg\!max}

\title[Positional Games]{Positional Games}

\author[Michael Krivelevich]
{Michael Krivelevich}

\begin{document}

\begin{abstract}
Positional games are a branch of combinatorics, researching a variety of two-player games, ranging from popular recreational games such as Tic-Tac-Toe and Hex, to purely abstract games played on graphs and hypergraphs. It is closely connected to many other combinatorial disciplines such as Ramsey theory, extremal graph and set theory, probabilistic combinatorics, and to computer science. We survey the basic notions of the field, its approaches and tools, as well as numerous recent advances, standing open problems and promising research directions.
\end{abstract}

\begin{classification}
Primary 05C57, 91A46; Secondary 05C80, 05D05, 05D10, 05D40.
\end{classification}

\begin{keywords}
Positional games, Ramsey theory, extremal set theory, probabilistic intuition.
\end{keywords}

\maketitle

\section{Introductory words}
Positional games are a combinatorial discipline, whose basic aim is to provide a mathematical foundation for analysis of two-player games, ranging from popular recreational games such as Tic-Tac-Toe and Hex to purely abstract games played on graphs and hypergraphs. Though the field has been in existence for several decades, motivated partly by its recreational side, it advanced tremendously in the last few years, maturing into one of the central branches of modern combinatorics. It has been enjoying mutual and fruitful interconnections with other combinatorial disciplines such as Ramsey theory, extremal graph and set theory, probabilistic combinatorics, as well as theoretical computer science.

The aim of this survey is two-fold. It is meant to provide a brief, yet gentle, introduction to the subject to those with genuine interest and basic knowledge in combinatorics. At the same time, we cover recent progress in the field, as well as its standing challenges and open problems. We also put a special emphasis on connections between positional games and other branches of combinatorics, in particular discussing the very surprising ubiquitous role of probabilistic intuition and considerations in the analysis of (entirely deterministic) positional games.

Due to obvious space limitations we will frequently be rather brief, omitting many of the proofs or merely indicating their outlines. More details, examples and discussions can be found in research monographs and papers on the subject.

\section{Basic setting and examples}\label{sec1}
{\em Positional games} involve two players alternately claiming unoccupied elements of a set $X$, the {\em board} of the game; the elements of $X$ are called vertices. Usually $X$ is assumed to be finite, although of course there are exciting infinite games to analyze. The focus of players' attention is a given family ${\cal H}=\{A_1,\ldots,A_k\}$ of subsets of $X$, called the {\em hypergraph of the game}; sometimes the members of ${\cal H}$ are referred to as  the {\em winning sets} of the game. In the most general version there are two additional parameters --- positive integers $p$ and $q$: the first player claims $p$ unoccupied elements in each turn, the second player answers by claiming $q$ vertices. (If in the very end of the game there are less unclaimed elements to claim than as prescribed by the turn of the current player, that player claims all remaining elements.) The parameters $p$ and $q$ define the {\em bias} of the game. The most basic case $p=q=1$ is the so called {\em unbiased} game. The game is specified completely by setting its outcome (first player's win/second player's win/draw) for every final position of the game, or more generally for every possible game scenario (an alternating sequence of legal moves of both players). For every game scenario there is only one possible outcome. Of course, the above definition is utterly incomplete and hence fairly vague. However, the accumulated research experience has shown that this is the right setting for the field. Depending on concrete game rules we get several game types, some of which are discussed later. For now let us state, using the standard game theory terminology, that positional games are two-player perfect information zero sum games with no chance moves.

Now we  illustrate the above general setting by providing several examples.

\noindent{\bf Example: Tic-Tac-Toe.} This is of course the first game that should come to anyone's mind. In our terminology, the board of the game $X$ is the 3-by-3 square. Two players, sometimes called Crosses and Noughts, claim in their turns one unoccupied element of the board each. The winning sets are three horizonal lines, three vertical lines, and two diagonals, all of size three; thus the game hypergraph ${\cal H}$ has eight sets of size three each, and is thus 3-uniform. (A hypergraph is called {\em $r$-uniform} if all its edges are of size $r$). The player completing a winning set first wins; if none of the lines is claimed by either player in the end, the game is declared a draw. Assuming optimal strategies of both players, the game is a draw,  as everybody knows; here the case analysis is essentially the only way to prove it.

\noindent{\bf Example: $n^d$.} This is a natural, yet extremely far reaching and challenging generalization of the classical Tic-Tac-Toe game. Given positive integers $d$ and $n$, the board $X$ of $n^d$ is the $d$-dimensional cube $X=[n]^d$, and the winning sets are the so-called {\em geometric lines} in $X$.
A geometric line $l$ is a family of $n$ distinct points $\left({\bf a}^{(1)},{\bf a}^{(2)},\ldots,{\bf a}^{(n)}\right)$ of $X$, where ${\bf a}^{(i)}=(a_1^{(i)},\ldots,a_d^{(i)})$, such that for each  $1\le j\le d$ the sequence of corresponding coordinates $(a_j^{(1)},a_j^{(2)},\ldots,a_j^{(n)})$ is either $(1,2,\ldots,n)$,  or $(n,n-1,\ldots,1)$, or constant (and of course at least one of the coordinates should be non-constant).
The winner is the player who occupies a whole line first, otherwise the game ends in a draw. The familiar Tic-Tac-Toe is $3^2$ in this notation. Another recreationally known instance of this game is  $4^3$, marketed by Parker Brothers under the name of Qubic.

This is a very complicated game, and resolving it for all pairs $(d,n)$ appears to be well out of reach. We can do it for the fairly simple two-dimensional case $d=2$ (first player's win for $n=1,2$, and a draw for $n\ge 3$), and also for the two extremes: $n$ large compared to $d$, and $d$ large compared to $n$. For the former case, we have the following:
\begin{theorem}\label{nd-draw}
If $n\ge 3^d-1$, then the $n^d$ game is a draw.
\end{theorem}
For the proof, one can argue that the maximum vertex degree in the game hypergraph of $n^d$ is $(3^d-1)/2$ for odd $n$  and is $2^d-1$ for even $n$; then a straightforward application of Hall's theorem shows that one can assign to every geometric line $l$ a pair of points $B(l)\subset l$ so that the assigned pairs are disjoint. Given such an assignment, either player can guarantee that he will not lose by claiming at least one element from each pair. This is a good example of the so called {\em pairing strategy} argument, frequently used to show the draw outcome in games of this type.

The case of fixed $n$ and large $d$ is first player's win. This is the famous Hales-Jewett theorem, and we will discuss it soon.

\noindent{\bf Example: Sim.} This too is a well known recreational game, whose mathematical description is as follows. The board of the game is the edge set of the complete graph $K_6$ on six vertices. Two players claim alternately one unoccupied edge of the board each, and the player completing a triangle of his edges first actually loses. Thus the game hypergraph ${\cal H}$ consists of the edge sets of all triangles in $K_6$, so ${\mathcal H}$ has $\binom{6}{3}=15$ sets of size 3 each. This is a reverse, or mis\'ere-type, game, the player completing a winning set first is the one to lose. Due to the standard fact $R(3,3)=6$, where $R(k,l)$ is the Ramsey number, the game cannot end in a draw. Sim has been solved by a computer and was proven to be a second's player win, with a complicated winning strategy.

\noindent{\bf Example: Hex.} This game was invented by the Danish
scientist Piet Hein in 1942 and was played and researched by John Nash in his student years. It is played on a
rhombus of hexagons of size $n\times n$ (in recreational
versions $n$ is usually 11), where two players, say, Blue and Red,
take the two opposite sides of the board each, and then alternately
mark unoccupied hexagons of the board with their own color. Whoever
connects his own opposite sides of the board first wins the game.

There is a catch here --- the game as described above does not fit our general scheme, as the players' winning sets are different. The problem can be cured by proving  that a player wins in Hex if and only if he prevents his opponent from winning, by blocking his winning ways. This is the so-called Hex theorem, known to be equivalent to the Brouwer fixed point theorem, see \cite{Gale79}. This allows to cast Hex in our general framework, by defining the winning sets to be the connecting sets of hexagons for the first player and declaring him the winner if he occupies one of them fully in the end; the second player's goal is redefined by assigning him instead the task of preventing the first player from claiming an entire winning set.
%

{\bf Example: Connectivity game.} The game
is played on the edge set of a multigraph $G$. The players, called Connector
and Disconnector, take turns in claiming one unoccupied edge of $G$
each,  with Disconnector moving first.
Connector wins the game if in the end the set of his edges contains a spanning tree of
$G$, Disconnector wins if he manages to leave
Connector with a non-connected graph.
Observe the highly non-symmetric goals of the players here. In our setting, the board is the set of edges of $G$, the winning sets are the edge sets of spanning trees in $G$; Connector wins the game if in the end he claims fully one of the winning sets, Disconnector wins otherwise. This game was treated by
Lehman, who proved:

\begin{theorem}[\cite{Leh64}]\label{Lehman_th}
If a multigraph $G$ has two edge-disjoint spanning trees, then
Connector wins the connectivity game played on $G$.
\end{theorem}
\noindent The proof is by induction on $|V(G)|$.
For the induction step, assume that $T_1$ and $T_2$ are edge-disjoint spanning trees
of $G$. If Disconnector claims an edge $e$ from, say, $T_1$, this move cuts $T_1$ into two
connected parts. Connector responds by claiming an edge $f$ of $T_2$
connecting these two parts, and then contracts $f$ by identifying
its two endpoints $u$ and $v$.
Then he updates $T_1$ and $T_2$ accordingly and applies induction to $G'$. The case where Disconnector claims an edge outside of
$T_1\cup T_2$ can be treated similarly.

Frequently the board of the game is the edge set of the complete graph $K_n$, and the players take turns in occupying its edges. Here is one such example.

{\bf Example. Hamiltonicity game.} This game is set up as follows. It is played by two players, who take turns claiming unoccupied
edges of  $K_n$, one edge each
turn. The first player wins if by the end of the game he manages to
make a Hamilton cycle (a cycle of length $n$) from his edges, while the
second player wins otherwise, i.e., if in the end he manages to put
his edge, or to break, into every Hamilton cycle. Here, the board is $E(K_n)$, and the winning
sets are $n$-cycles; the first
player wins if he claims an entire winning set by the end of the
game, and the second player wins otherwise.
This game was introduced and analyzed by Chv\'atal and Erd\H{o}s in
their seminal paper \cite{CE78}; it turned out to be an  easy
win for the cycle maker for all large enough $n$.

{\bf Example. Row-column game.} The board of the game is the
$n\times n$ square, and two players alternately claim its elements.
The goal of the first player is to achieve a sizable advantage
in some row or column; the question is how large an advantage he can get playing against a perfect opponent. We can place the game in our general framework as follows. For a
given parameter $k$, if the task is to decide whether the first player
can reach at least $k$ elements in one of the $2n$ lines of the
game, one can define the game hypergraph whose
board consists of  the $n^2$ cells of the square, and whose
winning sets are all $k$-subsets of the rows and columns.

If only rows (or only columns) are taken into account, then a simple pairing
strategy shows that the first player can achieve nothing for even
$n$, or 1 for odd $n$. However, when both rows and columns are
important, the first player can reach something of substance: Beck
proved \cite{Beck94} that he has a strategy to end up
with at least $n/2+32\sqrt{n}$ elements in some line. The upper
bound is due to Sz\'ekely, who showed \cite{Sze84} that the second
player can restrict his opponent to at most
$n/2+O(\sqrt{n\log n})$ elements in each line; the
gap of $c\sqrt{\log n}$ in the error term still stands.

\medskip

There is a crucial difference between casual games, where more experienced or skillful players have better chances to succeed, and formal games we consider here. We assume that both players have unbounded computational power and therefore play perfectly, choosing optimal moves each turn. Under this assumption, each positional game is {\em determined} and has exactly one of the three possible mutually exclusive outcomes: the first player has a winning strategy, the second player has a winning strategy, or both players have drawing strategies. A formal proof of this statement is easy and uses De Morgan's laws. Thus, solving a game means establishing its deterministic outcome out of the set of three possibilities. In order to do it formally, one can employ the fairly natural notion of a {\em game tree}. Given a positional game, its game tree is a rooted directed tree, where each node corresponds to a sequence of legal moves of both players, including the empty sequence
  --- the root of the tree, and there is an incoming edge to every legal sequence of moves from the sequence one move shorter. The leaves are exactly the final positions of the game. In order to find the outcome of the game, one can backtrack the game tree by labeling first its roots by the corresponding  game result, and then for each intermediate node marking it with the best possible move out of this position.
Although this simple procedure resolves every positional game, in reality it is usually extremely impractical, due to the huge size of the game tree. To illustrate this thesis, let us mention that the $4^3$ game, or the Qubic, is known to be first player's win, but according to Oren Patashnik, one of its solvers, the winning strategy fully spelled would be as thick as a phone book. This leaves us --- luckily, in fact --- with the necessity to develop general combinatorial tools for analyzing positional games.

Let us say a couple of words on where it all belongs
mathematically. The term ``positional games" can be somewhat
misleading. Classical game theory is
largely based on the notions of uncertainty and lack of perfect
information, giving rise to probabilistic arguments and the key
concept of a mixed strategy. Positional games, in contrast, are
perfect information games and as such can in principle be solved
completely by an all-powerful computer and hence are categorized as trivial in classical game theory.
This is not the case of course, due to the prohibitive complexity of the exhaustive search
approach; this only stresses the importance of accessible
mathematical criteria for analyzing such games. A probably closer
relative is what is sometimes called ``combinatorial game theory",
popularized by Conway and others, which includes such games as
Nim; they  are frequently based on algebraic arguments and
notions of decomposition. Positional games are usually quite
different and call for combinatorial arguments of various sorts.

Now we wish to mention a few papers and researchers who were influential in the development of the field. Of course, the choice below is rather subjective. The paper of Hales and Jewett \cite{HJ63} has established a tight connection between positional games and Ramsey theory. As is the case with many combinatorial disciplines, Paul Erd\H{o}s was a  pioneer here. First he wrote the paper \cite{ES73} with Selfridge, where potential functions were used for game analysis, at the same time providing the first derandomization argument,  a crucial notion in the theory of algorithms. Then Chv\'atal and Erd\H{o}s \cite{CE78} introduced biased games. J\'ozsef Beck shaped the field for the last three decades with his many papers that were instrumental in turning positional games into a coherent combinatorial discipline.

We complete this introductory section with suggestions for further reading. The fundamental monograph \cite{Beckbook} of  Beck  covers many facets of positional games. His more recent book \cite{Beckbook2} contains a lot of material on games too. The new book \cite{HKSSbook} can serve as a gentle introduction to the subject, at the same time covering  recent important developments.

\section{Strong games}
Strong games are probably the most natural, at least from the layman perspective, type of games. A {\em strong game} is played on a game hypergraph $(X,{\mathcal H})$ by two players, called First Player or {\tt FP}, and Second Player, or {\tt SP}, who take turns in occupying previously unclaimed elements of $X$, one element each time; First Player starts. The winner is the {\em first} player to occupy completely a winning set $A\in {\mathcal H}$; if this has not happened till the end of the game, it is declared a draw. Tic-Tac-Toe and $n^d$ are games of this type.

The childhood intuition suggesting it is beneficial to be the player to move first has a solid mathematical basis:

\begin{theorem}\label{SS}
In a strong game played on $(X,{\mathcal H})$, First Player can
guarantee at least a draw.
\end{theorem}

\begin{proof} The proof applies the so-called {\em strategy stealing principle}, observed by Nash. Assume to the contrary that Second Player  has a winning
strategy ${\cal S}$. The strategy is a complete recipe, prescribing
{\tt SP} how to respond to each move of his opponent, and to
reach a win eventually. Now, First Player ``steals" ${\cal S}$ and adopts it as follows. He starts with an
arbitrary move and then pretends to be Second Player, by ignoring
his first move. After each move of {\tt SP}, {\tt FP} consults ${\cal S}$ and responds accordingly. If he is told to
claim an element of $X$ which is still free, he does so;
if this element has been taken by him as his previously ignored
arbitrary move, he takes another arbitrary move instead. Observe that an extra move can only benefit
First Player. Since ${\cal S}$ is a winning strategy, at some point  {\tt FP} claims fully a winning set, even ignoring his
extra move, before {\tt SP} was able to do so. It follows that First Player has a winning strategy, excluding the possibility that Second Player has a winning strategy and thus providing the desired contradiction.
\end{proof}
Thus, in any strong game there are only two possible outcomes: First Player's win or a draw. (Both of them happen indeed for particular games.) This is perhaps the single most powerful result in positional games --- it is valid for every strong game! At the same time, it is rather useless, since due to the inexplicit nature of the proof it provides no concrete directions for {\tt FP} to reach at least a draw.

There are many games for which draw is impossible. This is usually a Ramsey-type statement: if for any two-coloring of $X$ (the colors are the moves of corresponding players) there is a monochromatic winning set $A\in{\mathcal H}$, then there is no final drawing position. We have:
\begin{corollary}\label{strong-Ramsey}
If in a strong game played on $(X,{\mathcal F})$ there is no final
drawing position, then First Player has a winning strategy.
\end{corollary}

For example, we conclude that the clique game $(K_n,K_q)$ (the board is the edge set of the complete graph $K_n$, the player completing a clique $K_q$ first wins) is First Player's win for $n\ge R(q,q)$, since by the definition of Ramsey numbers, there is no final drawing position here. Another example is the game of Hex;  Nash, in a pioneering application of the strategy stealing principle, observed that due to the board symmetry strategy stealing is applicable, and the first player wins the game. Still, no clue as for how exactly the first player should play to win in these cases...

The most inspiring instance of application of this pair (Strategy Stealing, Ramsey) is probably for the $n^d$ game. Recall that we stated that this game is a draw for $n$ large enough compared to $d$ (Theorem \ref{nd-draw}). In the opposite direction, Hales and Jewett, in one of the cornerstone papers of modern Ramsey theory \cite{HJ63}, proved:
\begin{theorem}\label{HJ-thm}
For every $k$ and $n$ there exists $d_0=d_0(k,n)$ such that for every $d\ge d_0$ every $k$-coloring of $X=[n]^d$ contains a monochromatic geometric line.
\end{theorem}
\noindent(See \cite{Shelah-HJ}, \cite{Nilli-HJ} for simpler proofs/better bounds on $d_0$.) The Hales-Jewett theorem is of course a Ramsey-type result, but it implies immediately the following nice corollary, which was apparently the original motivation behind \cite{HJ63}:
\begin{corollary}\label{nd-win}
For every $n$ there exists $d_0=d_0(n)$ such that for every $d\ge d_0$ the game $n^d$ is First Player's win.
\end{corollary}
\noindent To see this, simply apply Corollary \ref{strong-Ramsey} and Theorem \ref{HJ-thm} with $k=2$.

Regretfully our story about strong games nearly ends here. The above two main tools (strategy stealing, Ramsey-type arguments) exhaust our set of general tools available to handle these games. In addition, strategy stealing is very inexplicit, while Ramsey-type statements frequently provide astronomic bounds. So the situation leaves a lot to be desired. At the present state of knowledge we are unable to resolve even most basic games. The inherent difficulty in analyzing strong games can be explained partially by the fact that they are not hypergraph monotone. By this we mean the existence of examples (pretty easy ones in fact, see, e.g., Ch. 9.4 of \cite{Beckbook2}) of game hypergraphs ${\mathcal H}$ which are wins for First Player, yet one can add an extra set $A$ to ${\mathcal H}$ to obtain a new hypergraph ${\mathcal H}'$ which is a draw; this is what Beck calls the {\em extra set paradox}, and it is indeed quite disturbing.

However not all is lost, and some very nice and surprising results about particular strong games have been obtained recently. We will cover them later.

\section{Maker-Breaker games}
We have established that in strong games it is beneficial to be the first player to move -- by Theorem \ref{SS} he can guarantee at least a draw. If so, and perhaps thinking more practically, the second player can lower his sights and play more defensively instead, aiming to prevent his opponent from occupying fully a winning set, or putting it differently, to ``break" into every winning set. This leads naturally to the very important notion of Maker-Breaker games. Given a hypergraph $(X,{\mathcal H})$, the {\em Maker-Breaker game} is defined as follows. There are two players, called now Maker and Breaker, taking turns in occupying one element of $X$ in each turn. We assume that Maker moves first, unless said otherwise. Maker wins if in the end of the game he has occupied fully a winning set $A\in{\mathcal H}$, Breaker wins otherwise, i.e., if he claims at least one element in every winning set.

Maker-Breaker games have certain similarities to strong games; Maker should probably be compared to First Player, and Breaker to Second Player. Observe, for example, that if Breaker wins against Maker on ${\mathcal H}$, then Second Player draws in the corresponding strong game, using the same strategy. However, these game types are more different than similar: Maker, unlike First Player, needs to occupy a winning set eventually, and not necessarily first, in order to win; there is no draw here.
Sometimes Maker-Breaker games are also called {\em weak games}, to contrast them with  strong games.

Going back to our examples from Section \ref{sec1}, we can classify some of them now as Maker-Breaker games. They are the connectivity game (Connector is our Maker --- we assumed him to move second, but this is a tiny detail), the Hamiltonicity game and the row-column game. Moreover, as we explained there, Hex can be put into this framework too.

We remark that the following intuitive statement is correct: if Maker wins the game played on ${\mathcal H}$ as the second player, then he also has a winning strategy as the first player; an analogous statement is of course valid for Breaker as well.

Let us now concentrate on the prospects of each of the players, and on tools available to argue for their corresponding sides. We start with Breaker. Breaker's win is closely related to the 2-colorability problem in hypergraphs (frequently also called Property B), popularized by Erd\H{o}s. Observe:
\begin{proposition}\label{2col}
If the Maker-Breaker game played on a hypergraph $(X,{\mathcal H})$ is Breaker's win, then ${\mathcal H}$ is 2-colorable.
\end{proposition}
\begin{proof} As we mentioned, Breaker's win as the second player guarantees his win as the first player as well. We can thus imagine the two-player game on ${\mathcal H}$ where both the first and the second players think of themselves as Breaker and follow Breaker's winning strategy for the corresponding player. Thus each of them comes out as a winner, meaning that in the end each edge $A$ of ${\mathcal H}$ will carry the marks (the colors) of both players. It follows that ${\mathcal H}$ is 2-colorable.
\end{proof}
It is customary nowadays to use random coloring to argue that a hypergraph is 2-colorable. It is perhaps much less standard to  realize that there is a game theoretic way to look at this problem. The following statement is easily proven through the usual random argument. Let $(X,{\mathcal H})$ be a hypergraph. If $\sum_{A\in{\mathcal H}}2^{-|A|}<1/2$, then ${\mathcal H}$ is 2-colorable. (Color each of the vertices of $X$ randomly and independently in red with probability 1/2 and in blue with probability 1/2; for each edge $A$ the bad event $E_A=$``$A$ is monochromatic" has probability $2^{-|A|+1}$, hence
$\text{Prob}[\bigcup_{A\in{\mathcal H}}E_A]\le\sum_{A\in{\mathcal H}}2^{-|A|+1}<1$, and thus there is  a 2-coloring of ${\mathcal H}$.) Erd\H{o}s and Selfridge \cite{ES73} provided a game strengthening of this result:
\begin{theorem}\label{ES-th}
Let $(X,{\mathcal H})$ be a hypergraph. If
\begin{equation}\label{ES-crit}
\sum_{A\in{\mathcal H}}2^{-|A|}<1/2\,,
\end{equation}
then ${\mathcal H}$ is Breaker's win.
\end{theorem}
\begin{proof} At any stage of the game the board $X$ is split into three sets: the set $M$ of vertices claimed by Maker, the set $B$ of vertices of Breaker, and the set $F$ of currently free vertices. Define the potential function $\Psi=\sum_{A\in{\mathcal H}:A\cap B=\emptyset} 2^{-|A\setminus M|}$.  Observe that if Maker occupies at some point of the game a winning set $A\in{\mathcal H}$ fully, then  $\Psi\ge 1$ at that point. Thus, for Breaker to win it is enough to maintain the value of $\Psi$  below 1 during the game. The initial value of the potential is less than $1/2$ by the assumption of the theorem; after the first move of Maker it increases by at most the factor of 2 and is thus less than 1, a good start. Hence it is enough to prove that Breaker has a strategy to ensure that after each round (a round here is Breaker's move, followed by Maker's move) the value of $\Psi$ does not increase. Suppose we are in the beginning
  of round $i$ with partition $X=M_{i-1}\cup B_{i-1}\cup F_{i-1}$ and potential $\Psi_{i-1}$. Breaker's choice $b_i$ is natural: he chooses to claim the element $b_i$ maximizing the potential's decrease: $$b_i=\argmax_b\sum_{{A\cap B_{i-1}=\emptyset}\atop{b\in A}}2^{-|A\setminus M_{i-1}|}\,.$$ If Maker then claims an element $m_i$, the updated value $\Psi_i$ of the potential is:
$$
\Psi_i=\Psi_{i-1}-\sum_{{A\cap B_{i-1}=\emptyset}\atop{b_i\in A}}2^{-|A\setminus M_{i-1}|}+\sum_{{A\cap B_{i-1}=\emptyset}\atop{m_i\in A}}2^{-|A\setminus M_{i-1}|}-\sum_{{A\cap B_{i-1}=\emptyset}\atop{b_i,m_i\in A}}2^{-|A\setminus M_{i-1}|}\le \Psi_{i-1}\,,
$$
due to the choice of $b_i$.\end{proof}
The following construction shows that criterion (\ref{ES-crit}) is tight. Let $X=\{c\}\cup\{l_1,\ldots,l_k\}\cup\{r_1,\ldots,r_k\}$. Define ${\mathcal H}=\{A\subset X: c\in A,\, |A\cap \{l_i,r_i\}|=1,\,i=1,\ldots, k\}$. Maker wins the game on ${\mathcal H}$ by first taking $c$, and then claiming the sibling of Breaker's move ($l_i$ for $r_i$, and  $r_i$ for $l_i$).

The proof of the Erd\H{o}s-Selfridge theorem is quite simple (not a bad thing in itself!), but the result is truly remarkable for a variety of reasons. First, it provides a concrete and  very useful criterion for Breaker's win. It also serves as an inspiring example of applying potential functions in positional games. Another important feature of the proof is that it supplies a simple polynomial (in $|X|+|{\mathcal H}|)$ algorithm for Breaker to win. Essentially the same argument can be used to derandomize the random 2-coloring argument given above. In fact, this was the first instance of the  {\em method of conditional probabilities}, an important  general approach to derandomizing randomized algorithms, one of the major topics in theoretical computer science, see, e.g., \cite{AS,MR}. Thus, the field of positional games reaches well beyond its immediate scope.

Here is an example of applying the Erd\H{o}s-Selfridge criterion. In the Maker-Breaker version of the clique game $(K_n,K_q)$, Maker wins if in the end of the game he has claimed a clique of size $q$. Let $q(n)$ be the largest $q$ for which Maker wins the game. We can argue that $q(n)\le 2\log_2n$. Indeed, the game hypergraph ${\mathcal H}$ has $\binom{n}{q}$ edges and is $\binom{q}{2}$-uniform. Hence, by (\ref{ES-crit}) if $\binom{n}{q}2^{-\binom{q}{2}}<1/2$, Breaker wins. Solving this inequality for $q=q(n)$ gives the claimed bound. For the lower bound on $q(n)$, recall that if  $n\ge R(q,q)$, then First Player wins in the corresponding strong game, and Maker can follow his footsteps. Plugging in the standard upper bound $R(q,q)<4^q$ we derive $q(n)\ge \frac{1}{2}\log_2n$. Beck discusses the asymptotically tight lower bound $q(n)\ge (2-o(1))\log_2n$ in his book \cite{Beckbook}.

For Maker's side, Beck proved \cite{Beck81} the following criterion.
\begin{theorem}\label{Maker}
Let $(X,{\mathcal H})$ be an $r$-uniform hypergraph. If $|{\mathcal H}|>2^{r-3}\cdot\Delta_2({\mathcal H})\cdot|X|$, where $\Delta_2({\mathcal H})=\max\{deg(x,y) : x\ne y\in X\}$ and $deg(x,y)=|\{A\in{\mathcal H}: x,y\in A\}|$, then ${\mathcal H}$ is Maker's win.
\end{theorem}
The proof is similar ideologically to that of Erd\H{o}s and Selfridge and uses an appropriately defined potential function. To illlustrate this criterion, consider the arithmetic progression game $W(n,s)$. This is a Maker-Breaker game, whose board is $[n]$, and Maker wins if in the end he claims an arithmetic progression of length $s$. Let $s(n)$ be the largest $s$ for which Maker wins the game $W(n,s)$. Since the number of arithmetic progressions of length $s$ in $n$ is easily seen to exceed $\frac{n^2}{4(s-1)}$, and each pair of elements $x\ne y\in [n]$ is contained together is at most $\binom{s}{2}$ arithmetic progressions, applying Theorem \ref{Maker} gives $s(n)\ge (1-o(1))\log_2n$. An application of the Erd\H{o}s-Selfridge bound for Breaker's side gives $s(n)< 2\log_2n+1$. The right asymptotic answer here is actually $s(n)=(1+o(1))\log_2n$ \cite{Beck81}.

\section{Biased games, threshold bias}
Many (unbiased) Maker-Breaker games are drastically in favor of Maker, and he wins them easily. Here are few such examples, where the board is the edge set of $K_n$, we assume $n$ to be large enough. In the triangle game, where Maker wins if in the end he creates a copy of the triangle $K_3$, Maker is easily seen to win in 4 moves. For the connectivity game, Maker wins as well, say, by Lehman's Theorem \ref{Lehman_th}. The {\em non-planarity game}, where Maker wins if in the end his graph is non-planar, is a complete no-brainer: by Euler's formula, a graph on $n$ vertices with more than $3n-6$ edges is non-planar, so Maker just waits to accumulate $3n-5$ edges to declare his victory. If so, it appears quite natural to change the rules of the game to level the field and to increase Breaker's chances to win. One obvious way to do it is to introduce the game bias, as proposed in the pioneering paper of Chv\'atal and Erd\H{o}s \cite{CE78}; this is the subject of this section. Another possible approach is to sparsify the board of the game, and we will discuss it later.

Here is a formal definition of a biased Maker-Breaker game. Let $m$ and $b$ be positive integers, and let $(X,{\mathcal H})$ be a hypergraph. The {\em biased $(m:b)$ Maker-Breaker game} $(X,{\mathcal H})$ is the same as the Maker-Breaker game $(X,{\mathcal H})$, except that Maker claims $m$ free board elements per move and Breaker claims $b$ elements. The numbers $m$ and $b$ are referred to as the {\em bias} of Maker and Breaker, respectively. The most frequently considered case is that of $m=1$.

To illustrate the definition, consider the biased $(1:b)$ {\em triangle game} ${\mathcal H}_{K_3,n}$, as was done in \cite{CE78}. For $b\le \sqrt{2n}-C$ for some $C>0$ Maker wins by first accumulating enough edges at a vertex $u$, and then by closing a triangle containing $u$. For $b\ge 2\sqrt{n}$, the game is Breaker's win --- in response to each move $e_i=(u_i,v_i)$ of Maker, Breaker claims free $b/2$ edges incident to $u_i$ and $b/2$ free edges incident to $v_i$, also blocking all immediate threats of Maker. The critical value of $b=b(n)$ for this game is still unknown, Balogh and Samotij \cite{BS11} improved recently Breaker's side to $b\ge (2-1/24)\sqrt{n}$.

If we are serious about biased Maker-Breaker games, we should probably start our systematic study from the simplest case where the winning sets of the game hypergraph are pairwise disjoint. This is the famous Box Game introduced by Chv\'atal and Erd\H{o}s \cite{CE78}. In a game Box$(p,q;a_1,\ldots,a_n)$ the board of the game $X$ is a union of pairwise disjoint sets (boxes) $A_1,\ldots,A_n$ of sizes $a_1,\ldots,a_n$, respectively, forming the game hypergraph. To pay homage to this important game, and also with future games with identity changes in mind, we call players BoxMaker and BoxBreaker. In each move BoxMaker removes $p$ elements from the boxes, and BoxBreaker destroys $q$ boxes of his choice in return. BoxMaker wins if in the end he manages to empty one of the boxes before it is destroyed by BoxBreaker. In the case where all $n$ boxes are of equal size $s$, we use the notation Box$(p,q;n\times s)$. This game was analyzed by Chv\'atal and Erd\H
 {o}s for the nearly uniform case, the analysis for the general case was performed by Hamidoune and Las Vergnas \cite{HL}.
\begin{theorem}\textnormal{\cite{CE78}}\label{Box-Maker}
If $s\le (p-1)\sum_{i=1}^{n-1}1/i$, then BoxMaker, as the first or the second player, wins \textnormal{Box}$(p,1;n\times s)$.
\end{theorem}
\begin{proof} Follows from a bit more general statement we give now. Let $a_1\le\ldots a_n\le a_1+1$. Define $f(n,p)$ by the following recursion: $f(1,p)=0$ and $f(n,p)=\left \lfloor \frac{n (f(n-1, p) + p)}{n-1} \right \rfloor$ for $n \geq 2$. If $\sum_{i=1}^na_i\le f(n,p)$, then BoxMaker, as the second player, wins Box$(p,1;a_1,\ldots,a_n)$. The proof proceeds by induction, for the inductive step BoxMaker in his current turn removes $p$ elements to keep the surviving boxes leveled. One can easily show that $f(n,p)\ge (p-1)n\sum_{i=1}^{n-1}1/i$.
\end{proof}
\begin{theorem}\textnormal{\cite{CE78}}\label{Box-Breaker}
If $s> p\sum_{i=1}^n1/i$, then BoxBreaker wins \textnormal{Box}$(p,1;n\times s)$.
\end{theorem}
\begin{proof} We give a proof from \cite{HKSSbook}. At any point of the game, denote the set of surviving boxes by $S$. BoxBreaker always destroys a box $i\in S$ of minimum size. Suppose by contradiction that BoxMaker wins the game at move $k$, $1\le k\le n$. W.l.o.g. assume that BoxBreaker destroys box $i$ in his $i$th move, and in his $k$th move BoxMaker fully claims box $k$. Let $c_i$ denote the remaining size of box $i\in S\cap \{1,\ldots,k\}$. Define now the potential function $\Psi$ by
$$
\Psi(j) :=\frac{1}{k-j+1}\sum_{i=j}^k c_i\,,
$$
the potential just before BoxMaker's move $j$. Then $\Psi(k)=c_k\le p$, as BoxMaker wins the game at move $k$, while $\Psi(1)=s$. In his $j$th move BoxMaker decreases  $\Psi(j)$ by at most $p/(k-j+1)$; in his $j$th move BoxBreaker destroys the smallest surviving box. Thus $\Psi(j+1)\ge \Psi(j)-p/(k-j+1)$. It follows that
$$
\Psi(k)\ge s-\left(\frac{p}{k}+\frac{p}{k-1}+\cdots+\frac{p}{2}\right)\ge s-p\left(\sum_{i=1}^n\frac{1}{i}-1\right)>p\,,
$$
a contradiction.
\end{proof}
\noindent We conclude that for the uniform Box Game Box$(p,q;n\times s)$, the game changes hands around $p=s/\ln n$.

Returning to general biased Maker-Breaker games, let us state a criterion for Breaker's win due to Beck \cite{Beck82}, sometimes called the biased Erd\H{o}s-Selfridge criterion.
\begin{theorem}\label{biasedBreaker}
Let $X$ be a finite set, let ${\mathcal H}$ be a family of subsets of $X$, and let $p$ and $q$ be positive integers.
If
\begin{equation}
\sum_{A \in {\mathcal H}} (1 + q)^{- |A|/p} < \frac{1}{1+q}\label{biasedES} \,,
\end{equation}
then Breaker has a winning strategy in the $(p:q)$ game $(X, {\mathcal H})$.
\end{theorem}
\noindent The proof, while certainly non-trivial, is  similar to that of Theorem \ref{ES-th}. There is an analog of Theorem \ref{Maker} for the biased case \cite{Beck82}, but we will not state it here.

Maker-Breaker games are bias monotone. By this we mean the following formal statement: if the $(m:b)$ Maker-Breaker game $(X, {\mathcal H})$ is Maker's win, then so is the $(m:(b-1))$ game. Maker just adapts his winning strategy for the $(m:(b-1))$ game, each time assigning an arbitrary fictitious $b$th element to Breaker after Breaker's move. This leads us to the following very important definition.
\begin{definition} \label{def::thresholdBias}
Let $(X,{\mathcal H})$ be a hypergraph such that $\min \{|A| : A \in {\mathcal H}\} \geq 2$. The unique positive integer $b_{\mathcal H}$ such that Breaker wins the Maker-Breaker $(1 : b)$ game $(X, {\mathcal H})$ if and only if $b \geq b_{\mathcal H}$ is called the \emph{threshold bias} of $(X, {\mathcal H})$.
\end{definition}
Determining or estimating the threshold bias of a game is a central goal of the theory of biased Maker-Breaker games, and is the main subject of this section. For the triangle game ${\mathcal H}_{K_3,n}$, it follows from our prior discussion that the threshold bias $b_{{\mathcal H}}$ is of order $\sqrt{n}$; determining its asymptotic value remains open.

Let us ask ourselves now: for natural biased games on the edge set of $K_n$, like positive minimum degree, connectivity, Hamiltonicity, etc., what are the values of the threshold bias? How do they compare between themselves? To the reader unexperienced in positional games these questions must appear very challenging, and even making an intelligent guess should be not so easy.

We start with Breaker's side. This is achieved through the following theorem of Chv\'atal and Erd\H{o}s \cite{CE78}.
\begin{theorem}\label{isol-vertex}
For every $\epsilon>0$, all large enough $n$ and $b\ge (1+\epsilon)\frac{n}{\ln n}$, Breaker can isolate a vertex in the $(1:b)$ Maker-Breaker game, played on $E(K_n)$.
\end{theorem}
\begin{proof}  Breaker first builds a clique $C$ of size $b/2$ such that all vertices of $C$ are isolated in Maker's graph. In his turn $i$ he locates two isolated vertices $u_i,v_i$ in Maker's graph, claims $(u_i,v_i)$ and then joins them completely to the current $C$, claiming more edges if needed. Maker in his turn can touch only one vertex of the clique. At the second stage, Breaker's goal is to isolate one of the vertices of $C$. Observe that all edges inside $C$ have already been taken by Breaker, thus the only relevant edges are those between $C$ and its complement $V\setminus C$; moreover, the edge sets $E_v=\{(u,v): v\in V\setminus C\}\,,v\in C$, are disjoint. Thus we can appeal to the box game Box$(b,1;\{E_v: v\in C\})$, where Breaker disguises himself as BoxMaker. Applying Theorem \ref{Box-Maker} we derive that Breaker can claim all edges of some $E_v$, isolating $v$ and winning the game.
\end{proof}
We conclude that the threshold bias for all games on $K_n$, where all winning sets of Maker are spanning graphs of positive minimum degree, is at most $(1+o(1))n/\ln n$. This might appear as a rather humble beginning, but the truth is that for quite many of them this is a tight estimate!

We now switch to Maker's side. Consider first the connectivity game. Already Chv\'tal and Erd\H{o}s showed, probably quite surprisingly, that the threshold bias for this game is of asymptotic order $n/\log n$. Here we present an argument of Beck \cite{Beck82}, providing also a better multiplicative constant.
\begin{theorem}\label{con-Maker1}
The threshold bias $b_{{\mathcal C}_n}$ for the connectivity game ${\mathcal C}_n$ on $K_n$ satisfies: $b_{{\mathcal C}_n}\ge (1-o(1))\frac{n}{\log_2n}$.
\end{theorem}
\begin{proof} Let $\epsilon>0$, and fix $b=b(n)=(1-\epsilon)n/\log_2 n$. We prove that Maker wins ${\mathcal C}_n$ playing against bias $b$. Observe that in order for Maker to win, it is sufficient (and also necessary) to put an edge into every cut $[S,\bar{S}]$ for $\emptyset\ne S\ne [n]$. So we see another change of roles here: Maker plays as (Cut) Breaker in the {\em cut game}. The board of the cut game is $E(K_n)$, and the winning sets are exactly the cuts $A_S=[S,\bar{S}]$. Applying criterion (\ref{biasedES}), we need to verify that
$$
\sum_{S}2^{-|A_S|/b}=\sum_{k=1}^{n/2}\binom{n}{k}2^{-k(n-k)/b}<\frac{1}{2}\,,
$$
which can be done through standard asymptotic manipulations, omitted here.
\end{proof}
What is then the asymptotic value of the threshold bias for the connectivity game and several related games? Which constant should we put in front of $n/\ln n$? Erd\H{o}s, with an amazing foresight, suggested the following very surprising solution. Suppose both Maker and Breaker in their $(1:b)$ game on $E(K_n)$, instead of being utterly clever and using perfect strategies, play {\em randomly}. Then the resulting Maker's graph is a random graph on $n$ vertices with $m$ edges for $m=\left\lceil\binom{n}{2}/(b+1)\right\rceil$, i.e., a graph drawn from the probability distribution $G(n,m)$. This puts us in the realm of random graphs, a very developed field where the understanding was far ahead that of positional games. We do not dwell on the background and known results in the theory of random graphs, referring the reader instead to its standard sources \cite{Bol-rgbook,JLR}. The relevant results are those about the thresholds for positive
 minimum degree, connectedness, and Hamiltonicity in $G(n,m)$. All three properties are known to appear typically at $m^*=\frac{1}{2}n\ln n$ (much more precise statements are available). This would translate to the threshold bias $b^*\approx\binom{n}{2}/m^*=n/\ln n$ for the random game. Now, the {\em Erd\H{o}s paradigm}, or the random graph intuition, suggested that for some biased Maker-Breaker games, like the connectivity game, the threshold bias for the perfectly played games should be asymptotically the same as for the entirely different random games. This approach bridges between two seemingly unrelated fields --- positional games and random graphs, and indicates the very important role of probabilistic considerations in completely deterministic games. A very bold conjecture --- which has proven to be true!

Now we state three recent results that established asymptotically the threshold bias for the connectivity game, the minimum degree $c$ game, and the Hamiltonicity game. The first two theorems are due to Gebauer and Szab\'o \cite{GS09}, the third is due to the author \cite{K11}.
\begin{theorem}\label{con-GS}
For every fixed $\epsilon>0$ and all sufficiently large $n$, if $b=(1-\epsilon)\frac{n}{\ln n}$, then Maker wins the $(1:b)$ Maker-Breaker connectivity game ${\mathcal C}_n$.
\end{theorem}

\begin{theorem}\label{mindeg-GS}
For every fixed $\epsilon>0$ and every fixed  positive integer $c$, for all large enough $n$, if $b=(1-\epsilon)\frac{n}{\ln n}$, then Maker can build a spanning graph of minimum degree at least $c$ in the $(1:b)$ Maker-Breaker game played on $E(K_n)$.
\end{theorem}

\begin{theorem}\label{Ham-K}
For every fixed $\epsilon>0$ and all sufficiently large $n$, if $b=(1-\epsilon)\frac{n}{\ln n}$, then Maker wins the $(1:b)$ Maker-Breaker Hamiltonicity game played on $E(K_n)$.
\end{theorem}
\noindent Recalling Theorem \ref{isol-vertex}, we conclude that the threshold bias for each of the three games above is asymptotic to $n/\ln n$, completely in line with the Erd\H{o}s paradigm!

We will not say much about the proofs of the above theorems, referring the reader instead to the original papers. Let us mention that the proofs of the first two theorems go back to basics --- Maker reaches his goal directly, instead of using dual approaches and descriptions as in the proof of Theorem \ref{con-Maker1}. Cleverly devised potential functions are used in both of the proofs. The third proof uses a modification of the second result and its proof. It turns out that for the minimum degree $c$ game, Maker has a strategy to reach degree $c$ at any vertex $v$ before Breaker accumulates $(1-\delta)n$ edges at $v$, for some $\delta=\delta(\epsilon)>0$. The strategy of Maker, as given by \cite{GS09}, points Maker to a vertex $v$ (specified by the current situation on the board) and tells him to claim an {\em arbitrary} free edge incident to $v$. The crucial twist is to use {\em randomness} here and to choose instead a {\em random} free edge at $v$, out of at least $\delta n$ edges available. One can argue that following this random strategy, with positive probability Maker can create a pretty strong expander in linearly many moves. At this point the deterministic nature of the game comes to our help --- the game considered is of perfect information, and thus winning with positive probability against a given strategy of Breaker means there is a deterministic (but unspecified) strategy to win. Returning to the Hamiltonicity game, Maker then quickly turns his expander into a connected graph and finally augments his connected expander to a Hamiltonian graph in a linear number of moves; here the proof uses fairly standard techniques from the theory of random graphs. Altogether, Maker wins the game in at most $18n$ moves, when the board is still mostly empty.

Now we cite another important result about biased Maker-Breaker games, due to Bednarska and \L uczak \cite{BL00}. For a given graph $H$, Maker wins the {\em $H$-game} played on the edges of a host graph $G$ if in the end he possesses a copy of $H$. For the case $H=K_3$ and $G=K_n$ we get the above treated triangle game. Define now the maximum 2-density $m_2(H)$ of $H$, a frequently used notation in random graphs, by
$$
m_2(H)=\max_{H_0\subseteq H, |V(H_0)|>2}\frac{|E(H_0)|-1}{|V(H_0)|-2}\,.
$$
\begin{theorem}\label{BL}
Let $H$ be a graph with at least three non-isolated vertices. The threshold bias for the  Maker-Breaker $H$-game on $E(K_n)$ satisfies: $b=\Theta\left(n^{1/m_2(H)}\right)$.
\end{theorem}
The proof of Maker's side uses a random strategy again, this time in the simplest possible form: Maker chooses each time a random edge to claim. This shows yet again that probabilistic considerations and arguments are ubiquitous in positional games, quite a surprising phenomenon.

Yet another connection between positional games and randomness was revealed in \cite{FKN14}, where Maker's win in certain biased games was achieved through Maker creating a random graph with few edges deleted at each vertex, and then invoking results about local resilience of random graphs \cite{SV08}.

We conclude this section with a very entertaining argument of Beck \cite{Beck02}, providing a lower bound for the Ramsey number $R(3,t)$ through biased Maker-Breaker games. The bound obtained is superseded by the best possible bound $R(3,t)=\Omega(t^2/\log t)$ of Kim \cite{Kim95} (see recent \cite{BK13}, \cite{FGM13} for better constants), but it matches the best bound known for long 35 years, obtained through various approaches \cite{Erdos61, S77, ESW95,K95}.
\begin{theorem}\label{Beck-Ramsey}
There exists a constant $c>0$ such that $R(3,t)\ge ct^2/\ln^2t$.
\end{theorem}
\begin{proof} Set $b=2\sqrt{n}$. Imagine two players playing on the edges of $K_n$. The first player, taking $b$ edges at a time, thinks of himself as Breaker in the $(1:b)$ Maker-Breaker game, whose goal is to prevent Maker from claiming a triangle. The second player, claiming one edge per move, thinks of himself as Breaker in a different game, namely, the $(b:1)$ Maker-Breaker game, whose goal is to claim an edge in every vertex subset of cardinality $t=C\sqrt{n}\ln n$. The first player wins his game by our analysis of the triangle game. The second player is victorious too for large enough $C$ --- this can be derived by applying the biased Erd\H{o}s-Selfridge criterion (\ref{biasedES}) (the calculations are omitted). The result of the game, or perhaps of the games, is hence a partition of  $E(K_n)$ into two graphs, where the first graph is triangle-free, and the second graph does not have a clique of size $t$. It thus follows that $R(3,t)>n$.
\end{proof}

\section{Avoider-Enforcer games}
Recall the game of Sim described in Section \ref{sec1}. It has the interesting feature --- the player who occupies a winning set first actually loses. This is an example of {\em reverse}, or mis\'ere-type, games. Games of this type are the subject of this section. Reverse games are certainly interesting for their own sake, but for those who seek additional motivation to consider them, we now give an example of a Maker-Breaker game, where Maker relies on a reverse game to ensure his win.
\begin{theorem}\textnormal{\cite{HKSS08}}\label{MB-non-planar}
For every fixed $\epsilon>0$ and all sufficiently large $n$, if $b=(\frac{1}{2}-\epsilon)n$, then Maker wins the $(1:b)$ Maker-Breaker non-planarity game on $E(K_n)$.
\end{theorem}
\begin{proof} It follows  from Euler's formula that for  $k\ge 3$, if a graph $G$ on $n$ vertices has more than $\frac{k}{k-2}(n-2)$ edges and no cycles of length shorter than $k$, then $G$ is non-planar. Let $\alpha=\frac{\epsilon}{1-2\epsilon}$, and let $k$ be the smallest integer satisfying $1+\alpha>\frac{k}{k-2}$. In order the ensure his final graph is non-planar, it suffices for Maker to  {\em avoid} creating cycles shorter than $k$ in his first $(1+\alpha)n$ moves. This is an easy task --- the board will still have $\Theta(n^2)$ free edges after this number of moves, due to our choice of $\alpha$. Maker always chooses his next edge so as not to close a cycle of length less than $k$ and not to create a vertex of degree at least $n^{1/k}$; showing this is a feasible strategy is an easy exercise.
\end{proof}

We now define Avoider-Enforcer games formally. Let $a$ and $b$ be positive integers, and let $(X,{\mathcal H})$ be a hypergraph. In the {\em biased $(a:b)$ Avoider-Enforcer game} $(X,{\mathcal H})$ the two players are Avoider and Enforcer, with Avoider moving first and claiming exactly $a$ free elements in each turn, while Enforcer claims exactly $b$ free elements in each turn. Enforcer wins the game if he forces Avoider into claiming fully one of the sets $A\in{\mathcal H}$, and Avoider wins otherwise. The members $A$ of the game hypergraph ${\mathcal H}$ are sometimes called {\em losing sets}, to reflect the nature of the game.

Having seen the important role of the bias monotonicity and the threshold bias in Maker-Breaker games, we can hope that Avoider-Enforcer games behave similarly. However, this is pretty much {\em not} the case, as has been observed in \cite{HKS07}. Consider the following simple example: the game hypergraph ${\mathcal H}$ consists of two disjoint sets of size 2 each. It is immediate to check that for $a=b=1$ Avoider is the winner, for $a=1,b=2$ Enforcer wins, and for $a=b=2$ Avoider wins again; the example can easily be generalized to larger sets and bias numbers. This {\em lack of monotonicity} is a fairly disturbing feature, which prompted the authors of \cite{HKSS10} to adjust the definition in the following, rather natural, way: now Avoider claims {\em at least} $a$ elements in each turn, while Enforcer claims {\em at least} $b$ elements. This version is easily seen to be bias monotone, and for this reason we call this set of rules {\em monotone rules}, while the former set of rules is called {\em strict rules}. Each monotone Avoider-Enforcer game ${\mathcal H}$ has the threshold bias $b^{\normalfont\text{mon}}_{{\mathcal H}}$, which is the largest non-negative integer $b$ for which Enforcer wins the corresponding $(1:b)$ game. For strict rules, we can define instead the {\em lower threshold bias} $b^-_{{\mathcal H}}$ as the largest integer such that Enforcer wins the $(1:b)$ game for every $b\le b^-_{{\mathcal H}}$, and the {\em upper threshold bias} $b^+_{{\mathcal H}}$ as smallest integer such that Avoider wins the $(1:b)$ game for every $b> b^+_{{\mathcal H}}$.

Just like for Maker-Breaker games, determining or estimating the threshold bias(es) is a central task for Avoider-Enforcer games, under both sets of rules. The difference is that here the situation is frequently much more challenging, and our understanding of Avoider-Enforcer games does not quite match that of their Maker-Breaker counterparts.

As a warm-up example, consider the game ${\mathcal H}_{P_2,n}$, where Avoider aims to avoid creating a copy of the path $P_2$ of two edges. The biases for this game are as follows \cite{HKSS10}: $b^+_{{\mathcal H}_{P_2,n}}=\binom{n}{2}-2$, $b^-_{{\mathcal H}_{P_2,n}}=\Theta\left(n^{3/2}\right)$ and $b^{\normalfont\text{mon}}_{{\mathcal H}_{P_2,n}}=\binom{n}{2}-\left\lfloor\frac{n}{2}\right\rfloor-1$.

One may be tempted to think that for Avoider-Enforcer games on a game hypergraph ${\mathcal H}$, the monotone threshold bias $b^{\normalfont\text{mon}}_{{\mathcal H}}$ is always between the lower and the upper threshold biases  $b^-_{{\mathcal H}}$ and  $b^+_{{\mathcal H}}$. This is not the case as we will see shortly.

We now state several results obtained for both game rules. Let us start with the strict rules.
\begin{theorem}\textnormal{\cite{HKS07}}\label{AE-strict-conn}
For the Avoider-Enforcer connectivity game ${\mathcal C}_n$ on $E(K_n)$, played under strict rules, we have $b^-_{{\mathcal C}_n}=b^+_{{\mathcal C}_n}=\left\lfloor\frac{n-1}{2}\right\rfloor$.
\end{theorem}
This is a very unusual result in the sense that is provides the {\em exact} value of the threshold biases. Avoider's side $b^+_{{\mathcal C}_n}$ is trivial --- if Avoider ends up with less than $n-1$ edges, he just cannot create a spanning tree and is thus guaranteed to win. Enforcer's side uses the well-known fact that $K_n$ contains $\lfloor n/2\rfloor$ edge-disjoint spanning trees and the following theorem.
\begin{theorem}\textnormal{\cite{HKS07}}\label{AE-Lehman}
If $G$ contains $b+1$ pairwise edge-disjoint spanning trees, then Enforcer wins the $(1:b)$ Avoider-Enforcer connectivity game played on $E(G)$ under strict rules.
\end{theorem}
This theorem can be considered to be the Avoider-Enforcer analog of Lehman's Maker-Breaker Theorem \ref{Lehman_th}. Peculiarly enough, the biased Maker-Breaker version of Lehman does not go through: for any $k$ one can construct an example of a graph $G$ with $k$ edge-disjoint spanning trees, where Breaker wins the $(1:2)$-connectivity game. Theorem \ref{AE-Lehman} also readily implies the following result.
\begin{theorem}\textnormal{\cite{HKS07}}\label{AE-strict-k-conn}
For the $k$-connectivity Avoider-Enforcer game ${\mathcal C}^k_n$ on $K_n$, $k\ge 2$, one has: $\frac{n}{2k}\le b^-_{{\mathcal C}^k_n}\le b^+_{{\mathcal C}^k_n}\le\frac{n}{k}$.
\end{theorem}
For the Hamiltonicity game we understand Enforcer's side well under both rules:
\begin{theorem}\textnormal{\cite{KS08}}\label{AE-Ham}
If $b\le (1-o(1))\frac{n}{\ln n}$, then Enforcer has a winning strategy in the $(1:b)$ Hamiltonicity game on $E(K_n)$, under either strict or monotone rules.
\end{theorem}
Moving on to the monotone rules, we have the following key result.
\begin{theorem}\textnormal{\cite{HKSS10}}\label{AE-mindeg1}
If $b\ge (1+o(1))\frac{n}{\ln n}$, then Avoider has a strategy to be left with an isolated vertex in the monotone $(1:b)$ game on $E(K_n)$.
\end{theorem}
Theorems \ref{AE-Ham} and \ref{AE-mindeg1} combined establish the threshold bias $b^{\normalfont\text{mon}}$ of several important games (connectivity, perfect matching, Hamiltonicity, etc.) to be asymptotically equal to $n/\ln n$. Observe that for the strict connectivity game both threshold biases, which happen to be equal by Theorem \ref{AE-strict-conn}, are quite far from the threshold bias for the monotone version, perhaps contrary to our intuition.

Finally, Clemens et al. \cite{CEPT14} showed very recently:
\begin{theorem}
For $n$ large enough and $b\ge 200n\ln n$, Avoider has a strategy to be left with a graph with at most one cycle in the $(1:b)$ game on $E(K_n)$, under either strict or monotone rules.
\end{theorem}
It follows that the threshold biases under monotone rules, and the upper threshold biases  $b^+_{{\mathcal H}}$ under strict rules for the non-planarity and the non-$k$-colorability games with $k\ge 3$ are at most $200n\ln n$. 
Still, for these games we are quite far from nailing these biases.

For Avoider-Enforcer games with losing sets of constant size the situation is very far from satisfactory --- we know few sporadic results or bounds. For example, for the triangle game ${\mathcal H}_{K_3,n}$ the monotone threshold bias is $\Theta(n^{3/2})$ \cite{HKSS10}, very far from the threshold bias for the Maker-Breaker version; for the strict game we have: $b^-_{{\mathcal H}_{K_3,n}}=\Omega(n^{1/2})$,  $b^+_{{\mathcal H}_{K_3,n}}=O(n^{3/2})$ \cite{B14}. A recent paper of Bednarska-Bzd\c ega \cite{B14} contains several interesting results of this sort.

As for general tools to tackle Avoider-Enforcer games, we are somewhat shorthanded here. We have the following results:
\begin{theorem}\textnormal{\cite{HKS07}}
If $\sum_{A\in{\mathcal H}}\left(1+1/a\right)^{-|A|+a}<1$,
then Avoider wins the biased $(a:b)$ game ${\mathcal H}$, under both strict and monotone rules, for every $b\ge 1$.
\end{theorem}
\noindent(Observe the lack of sensitivity to the value of $b$ in the above criterion --- an obvious drawback.)
\begin{theorem}\textnormal{\cite{B14}}
Let $(X,{\mathcal H})$ be a hypergraph with all sets $A\in{\mathcal H}$ of size at most $r$. If
$\sum_{A\in{\mathcal H}}\left(1+b/(ar)\right)^{-|A|+a}<1$,
then Avoider wins the biased $(a:b)$ game ${\mathcal H}$, under both strict and monotone rules.
\end{theorem}
\noindent This criterion is handy for games with losing sets of small size.

\section{More boards, more games}
Most of the concrete games we have considered so far are played on the complete graph $K_n$. This does not have to be the case, and many games on sparser graphs are equally interesting. Also, sparsifying the game board (rather than turning to biased games) can be seen as an alternative approach to provide Breaker with higher chances to win standard graph games against Maker.

A very typical setting is games on random graphs. In this scenario, for a given game type, say Maker-Breaker Hamiltonicity, we first set up a probability space of graphs, say, the binomial random graphs $G(n,p)$, and then ask about the probability of generating a board which is a win of each of the players.

Here is a representative result of this sort. Denote $N=\binom{n}{2}$. Consider the random graph process $\tilde{G}=(G_i)_{i=0}^{N}$ on $n$ vertices, described as follows. Start with the empty graph $G_0$ with vertex set $[n]$, and then for $1\le i\le N$, form a graph $G_i$ by adding to $G_{i-1}$ a random missing edge. For a monotone graph property $P$ and a (random) graph process $\tilde{G}$, we define the {\em hitting time} $\tau(\tilde{G},P)$ as the minimal $i$ such that $G_i$, the $i$th graph of the process, possesses $P$. For a random graph process $\tilde{G}$, the hitting time $\tau(\tilde{G},P)$ becomes a random variable and we can study its typical behavior. Consider now the unbiased Hamiltonicity game $\mathcal{HAM}$ with Breaker moving first, and let $M_{\mathcal {HAM}}$ be the property ``Maker wins $\mathcal{HAM}$ on $G$". Breaker clearly wins for the empty graph $G_0$, and by the classical Chv\'atal-Erd\H{o}s result \cite{CE78}, Maker is
 the winner for $G_N=K_n$. So the hitting time  $\tau(\tilde{G},P)$ lies somewhere in between. Those familiar with the theory of random graphs can guess that the key to Maker's win will be the disappearance of vertices of small degree. Indeed, if $\delta(G)<4$, then Breaker, moving first, wins the game by claiming all but at most one edge at a vertex of minimum degree in $G$. Thus, $\tau(\tilde{G},M_{\mathcal{HAM}})\ge \tau(\tilde{G},\delta_4)$, where $\delta_4$ is the property of having minimum degree at least 4.
\begin{theorem}\textnormal{\cite{BFHK12}}
For a random graph process $\tilde{G}$ on $n$ vertices, with high probability $\tau(\tilde{G},M_{\mathcal{HAM}})= \tau(\tilde{G},\delta_4)$.
\end{theorem}
Thus, typically in the random graph process, Maker starts winning the Hamiltonicity game {\em exactly} at the moment the last vertex of degree less than 4 disappears. From this result one can derive, in a rather standard way, the corresponding result for $G(n,p)$ --- the threshold probability for Maker's win stands at $p(n)=\frac{\ln n+3\ln\ln n}{n}$. The same paper \cite{BFHK12} gets hitting time results also for the perfect matching and the $k$-connectivity Maker-Breaker games. Biased Hamiltonicity games on $G(n,p)$ were considered in \cite{FGKN14}, where it was shown that for $p\gg\frac{\log n}{n}$, the threshold bias $b_{\mathcal{HAM}}$ satisfies typically $b_{\mathcal{HAM}}=(1+o(1))\frac{np}{\log n}$. This fits very well with the probabilistic intuition of Erd\H{o}s, as was predicted by Stojakovi\'c and Szab\'o \cite{SS05} who initiated the study of positional games on random graphs.

The $H$-game, where Maker wins if he creates a copy of a fixed graph $H$, was considered in its unbiased version for the case of random boards in \cite{SS05}, \cite{MS14}, \cite{NSS14}. In particular, M\"uller and Stojakovi\'c proved in \cite{MS14}:
\begin{theorem}\label{MS-th}
Let $k\ge 4$ be fixed. There exists a constant $c=c(k)>0$ such that for $p\le cn^{-\frac{2}{k+1}}$, a random graph $G\sim G(n,p)$ is with high probability such that Breaker wins the unbiased Maker-Breaker $K_k$-game on $G$.
\end{theorem}
The matching result for starting Maker (if $p\ge Cn^{-\frac{2}{k+1}}$, then Maker typically wins the $K_k$-game on $G(n,p)$) can be derived from the general Ramsey-type result of R\"odl and Ruci\'nski \cite{RR95} and the strategy stealing argument, and it can be easily adapted for the case where Breaker starts. The case $k=3$ turns out to be different. There the threshold lies at $p=n^{-5/9}$ for a good local reason: Maker wins the triangle game on the graph $K_5-e$, hence Maker wins no later than this graph appears in $G(n,p)$ -- which typically happens at $p=n^{-5/9}$. The paper \cite{MS14} provides a hitting time version of this result.

Another sparsification-type approach asks for the minimal size of a game board, on which Maker can create a prescribed structure. For example, \cite{HKSS11} and \cite{BP12} proved that a graph $G$ on $n$ vertices on which Maker wins the positive minimum degree game has at least about $10n/7$ edges, and this estimate is tight. Gebauer, building partly on ideas from \cite{FK08}, showed:
\begin{theorem}\textnormal{\cite{Geb13}}\label{Geb-sizeR}
For every $d>0$ there exists $c=c(d)$ such that for every graph $H$ on $n$ vertices of maximum degree at most $d$, there is a graph $G$ with at most $cn$ edges such that Maker wins the unbiased $H$-game on $G$.
\end{theorem}
These problems can be viewed as game versions of size Ramsey numbers. For a graph $H$, the {\em size Ramsey number} $\hat{r}(H)$ is the smallest $M$ for which there exists a graph $G$ with $M$ edges such that any red-blue coloring of the edges of $G$ produces a monochromatic copy of $H$ (we say also that $G$ arrows $H$). If $G$ arrows $H$, then by the strategy stealing argument (again!), Maker as the first player wins the $H$-game on $G$.  Hence size Ramsey numbers upper bound their game counterparts. The so obtained bounds are not always tight: for example, R\"odl and Szemer\'edi showed \cite{RS00} the existence of a graph $H$ on $n$ vertices of maximum degree 3 and size Ramsey number $\hat{r}(H)\ge cn\log^{\frac{1}{60}}n$, thus making it impossible to derive Theorem \ref{Geb-sizeR} from general size Ramsey results.

In many games, the identity of the winner is easy to establish, and one can then ask how long it takes him to win. For example, Lehman's Theorem \ref{Lehman_th} and its proof show that if a graph $G$ on $n$ vertices has two edge-disjoint spanning trees, then Maker, as the first or the second player, wins the connectivity game on $G$ in $n-1$ moves; this is clearly optimal. Define  the {\em move number} $\text{move}({\mathcal H})$ of a weak (resp. strong) game ${\mathcal H}$ as the smallest $t$ such that Maker (respectively {\tt FP}) has a strategy to win ${\mathcal H}$ within $t$ moves. If Breaker is the winner ({\tt SP} draws, resp.), we set $\text{move}({\mathcal H})=\infty$. The move number for  Avoider-Enforcer games is defined accordingly. Here is a sample of results about this concept. \cite{HKSS09-2} showed that for sufficiently large $n$, playing on the edges of $K_n$, Maker can create a Hamilton path in $n-1$ moves, and a perfect matching (assuming $n$ is even) in $n/2+1$ moves. Hefetz and Stich proved:
\begin{theorem}\textnormal{\cite{HS09}}\label{HS-Ham}
Let ${\mathcal HAM}$ denote the Maker-Breaker Hamiltonicity game played on the edges of $K_n$. Then $\textnormal{move}({\mathcal HAM})=n+1$, for all sufficiently large $n$.
\end{theorem}
For the Maker-Breaker clique game $(K_n,K_q)$, the move number is known to be between $2^{\frac{q}{2}}$ \cite{Beck02} and $2^{\frac{2q}{3}}\text{poly}(q)$ \cite{Geb12}, for all large enough $n$. For Avoider-Enforcer games, Anuradha et al. \cite{AJSS08} proved that Avoider can stay planar for as long as $3n-26$ turns when playing against Enforcer in the unbiased game on $K_n$, a constant away from the trivial upper bound of $3n-6$.

Fast wins are closely related to strong games. Observe that if in a game ${\mathcal H}$ all winning sets have cardinality $n$, and in the weak game on ${\mathcal H}$ Maker has a strategy to win in $n$ moves, then in the strong game on ${\mathcal H}$ First Player wins in $n$ moves. Thus for example we can derive that {\tt FP} wins the Hamilton path game on $K_n$. This scenario is quite rare though. Still, in a recent exciting development, Ferber and Hefetz were able to use fast winning strategies in Maker-Breaker games to analyze much harder strong games.
\begin{theorem}\label{FH-strong}\textnormal{\cite{FH11,FH14}}
Let $\mathcal{PM}$, $\mathcal{HAM}$, $\mathcal{C}^k$ denote the perfect matching, Hamiltonicity and spanning $k$-connectivity games, resp., played on the edges of $K_n$. The strong versions of these games are First Player's win, for large enough $n$. Also, $\textnormal{move}(\mathcal{PM})\le \frac{n}{2}+2$, $\textnormal{move}(\mathcal{HAM})\le n+2$, and $\textnormal{move}(\mathcal{C}^k)=\left\lfloor\frac{kn}{2}\right\rfloor+1$ for $k\ge 3$.
\end{theorem}
The key to all these proofs is analysis of fast strategies for the corresponding Maker-Breaker games and their adaptation for strong games. A natural outcome of this proof approach is explicit winning strategies for {\tt FP}. These results provide a new lease of life for the whole subject of strong games, notorious for its difficulty.

Discrepancy games can be viewed as a hybrid between Maker-Breaker and Avoider-Enforcer games. In such games, the first player, called Balancer, aims to end up with the correct proportion of elements in each winning set. For example, in the unbiased case Balancer strives to get about half of his elements (not much more, not much less) in every winning set. These games have been considered in \cite{AKSS05}, \cite{HKS07-2}, see also Chapters 16, 17
of \cite{Beckbook}. \cite{FKPS05} describes a strategy for Balancer to construct a pseudo-random graph; since pseudo-random graphs are known to have many nice features \cite{KS06}, this result guarantees Maker's win in a variety of games.

There are also games involving directed/oriented graphs, or graph orientation. For example, in the tournament game a tournament $T$ on $k$ vertices is given, and Maker and Breaker take turns in claiming free edges of $K_n$, one edge each, with Maker also orienting his edges. Maker wins the game if his graph contains in the end a copy of $T$. Clemens, Gebauer, and Liebenau proved \cite{CGL14} that for $k=(2-o(1))\log_2n$ Maker can create a copy of any given $k$-vertex tournament $T$. This is asymptotically optimal and resolves a problem posed by Beck in \cite{Beckbook}.

Finally, let us describe yet another quite interesting class of games, not very well studied as of yet. In these games, the players are Picker and Chooser. Picker in his $i$th turn picks two free elements $v_i,v_i'$ of the board $X$ and presents them to Chooser --- who chooses one of these elements, with the other one going to Picker. In the {\em Chooser-Picker} version of the game, Chooser wins if in the end he completes a winning set $A\in{\mathcal H}$, and Picker wins otherwise; in the {\em Picker-Chooser} version Picker wins if he occupies a winning set.
These games, especially the Chooser-Picker variant, appear to be similar to Maker-Breaker versions, and it was even conjectured that if Breaker wins the game on ${\mathcal H}$, then Picker, whose job appears to be even easier, wins the Chooser-Picker game on  ${\mathcal H}$; however, this was disproved in \cite{Knox12}. Still,  Bednarska-Bzd\c ega showed \cite{B13} that the Erd\H{o}s-Selfridge criterion (\ref{ES-crit}) is also a winning criterion for Picker in Chooser-Picker games. Beck in \cite{Beckbook} analyzed the Picker-Chooser clique game $(K_n,K_q)$ with an amazing degree of precision, proving in particular that Chooser starts winning at $q=(1-o(1))2\log_2n$. One can also consider the biased versions of both game types, where at each round Picker picks between $1+p$ and $p+q$ free elements of the board, Chooser keeps $p$ of them, and the rest goes to Picker. See \cite{Beckbook}, \cite{CMP09}, \cite{C10}, \cite{B13} for more results.

\section{Open problems and challenges}

In this section we present a representative sample of open problems in the field, and also indicate promising research directions.

For strong games, there are still many more problems than answers, though the situation is more hopeful now after the recent results stated in Theorem \ref{FH-strong}. One concrete example is the following question:
\begin{problem}
Show that for every positive $q$ there exist $t$ and $n_0$ such that for every $n\ge n_0$ First Player can win in at most $t$ moves the strong clique game $(K_n,K_q)$.
\end{problem}
In other words, we are essentially asking for an {\em explicit} winning strategy for {\tt FP} in the clique game. As we stated before, for $n\ge R(q,q)$ First Player wins due to strategy stealing, but this is highly inexplicit. The problem appears to be open even for the case $q=5$.

For weak (Maker-Breaker) games, we are in a much better shape. Still there are many nice problems to tackle. In the {\em degree game}, a graph $G$ of minimum degree $d$ is given, and Maker aims to create a graph of highest possible minimum degree. Since the edges will be split evenly between the players in the end, the best Maker can hope for is $d/2$. Getting around $d/4$ is fairly easy, here is a sketch. Assume for simplicity all degrees in $G$ are divisible by 4. Using an Eulerian orientation, one can split the edges of $G$ between its vertices so that a vertex $v$ gets a set $E_v$ of $d(v)/2$ incident edges assigned to it. Maker then plays a pairing game on the disjoint sets $E_v$, answering Breaker's move $e\in E_v$ by $e'\in E_v$, and thus guaranteeing a quarter degree at every vertex.   At present, $d/4$ is the best known result,  even improving it to $(1/4+\epsilon)d$ would be quite nice.   If we allow $d$ to depend on the number of vertices $n=|V(G)|$, then for $d\gg
 \log n$ we can get $(1/2-o(1))d$ in the degree game using known discrepancy results, say \cite{AKSS05}.

Here is a very cute problem due to Duffus, \L uczak, and R\"odl  \cite{DLR98}.
\begin{problem}
Prove that for integers $b\ge 1$ and $r\ge 3$, there exists $C=C(b,r)$ such that for every graph $G=(V,E)$ of chromatic number at least $C$, Maker has a strategy to create a graph $M$ of chromatic number at least $r$ when playing the $(1:b)$ biased Maker-Breaker game on $E(G)$.
\end{problem}
\noindent(Duffus et al.~\cite{DLR98} asked actually the equally interesting vertex version of this problem, where players claim vertices of $G$ rather than edges;  we state it in the edge version, more in line with the prevailing setting of Maker-Breaker games.) The unbiased case $b=1$ is easy -- one can take $C=(r-1)^2+1$, and the argument goes as follows. Recall first that for every edge decomposition $E(G)=E(M)\cup E(B)$ one has $\chi(G)\le \chi(M)\chi(B)$, where $\chi(G)$ denotes the chromatic number of $G$. Hence, if for a graph $G$ with $\chi(G)>(r-1)^2$  Breaker has a strategy to prevent Maker from reaching $\chi(M)\ge r$, then that very strategy guarantees Breaker a graph $B$ with $\chi(B)\ge r$. Maker, who is assumed to move first, can then steal this strategy and achieve $\chi(M)\ge r$. However, even the next case $b=2,r=3$ is open. Only very partial results are available, see \cite{AHK10,FGKLPVV13}. 

The {\em neighborhood conjecture} is one of the most important problems in positional games, with several ramifications in other combinatorial and computer science questions. Recall that the Erd\H{o}s-Selfridge  criterion (\ref{ES-crit}) guarantees that a $k$-uniform hypergraph ${\mathcal H}$ with less that $2^{k-1}$ edges is Breaker's win, and is thus 2-colorable by Proposition
\ref{2col}. As we mentioned, this is tight for games. The obvious drawback of this result is that it does not take into account the local structure of $\mathcal{H}$, hence taking any non-empty Breaker's win ${\mathcal H}$ and repeating it enough times will eventually create a hypergraph  violating the Erd\H{o}s-Selfridge condition (still Breaker's win of course). A local version of Breaker's winning criterion would be:
 \begin{problem}
Determine
$$
D(k) := \min\{d: \exists\textnormal{ $k$-uniform Maker's win ${\mathcal H}$ with }\Delta({\mathcal H})\le d\}\,,
$$
\end{problem}
\noindent where $\Delta({\mathcal H})$ is the maximum degree of ${\mathcal H}$. The definition of $D(k)$ implies in particular that any $k$-uniform  ${\mathcal H}$ with $\Delta({\mathcal H})< D(k)$ is 2-colorable. The relation to the famous Lov\'asz Local Lemma \cite{EL75} is apparent --- a standard application of the Local Lemma gives that if  $\Delta({\mathcal H})\le\frac{2^{k-1}}{ek}$, then ${\mathcal H}$ is 2-colorable. The gap between known lower and upper bounds for $D(k)$ is truly astonishing. For the lower bound, we know the trivial $D(k)\ge \frac{k}{2}+1$  (pairing strategy). Gebauer proved \cite{Geb12-2} that $D(k)\le (1+o(1))\frac{2^k}{k}$, thus disproving the original neighborhood conjecture, stated in Beck's book \cite{Beckbook}. This bound was improved further to $D(k)\le (1+o(1))\frac{2^k}{ek}$ by Gebauer, Szab\'o and Tardos \cite{GST11}. The same paper \cite{GST11} reveals exciting connections between the neighborhood conjecture, the Local Lemma, and the  satisfiabilty problem, most central in computer science.

Another problem about Maker-Breaker games due to Beck \cite{Beckbook} is as follows.
 \begin{problem}
For the $(m:m)$ Maker-Breaker clique game on $K_n$, what is the largest clique size $q$ Maker is guaranteed to achieve?
\end{problem}
For the unbiased case $m=1$ the answer is $q=(1-o(1))2\log_2n$, due to Beck \cite{Beckbook}. This matches the probabilistic intuition very well, as the clique number of the random graph $G(n,1/2)$ is  typically equal to $(1-o(1))2\log_2n$. If so, one can expect that for $m>1$ the answer should be similar. This is however not the case for $m\ge 6$, as shown in \cite{Geb12}.

For Avoider-Enforcer games, the current state of affairs does not quite match the situation for their Maker-Breaker analogs. In particular, we do not know yet to resolve the strict Hamiltonicity game:
 \begin{problem}
What are the threshold biases $b^-$ and $b^+$ for the $(1:b)$ Avoider-Enforcer Hamiltonicity game on $E(K_n)$, played under strict rules?
\end{problem}
The gap between the lower bound $b^-\ge (1-o(1))n/\ln n$ from Theorem \ref {AE-Ham} and the trivial upper bound $b^+\le n/2-1$ is quite annoying. It would be also very nice to develop a general theory of Avoider-Enforcer $H$-games, for fixed $H$; so far these games have mostly been attacked on a game-to-game basis.

Finally, we mention Chooser-Picker and Picker-Chooser games. They are largely an uncharted territory, and natural and attractive problems abound there.

\bigskip
\noindent{\bf Acknowledgement.} The author wishes to thank Asaf Ferber, Dan Hefetz, Milo\v{s} Stojakovi\'{c} and Tibor Szab\'o for careful reading of the manuscript and many helpful comments, and also for an extensive and fruitful cooperation in research on this fascinating subject.


\begin{thebibliography}{99}
\bibitem{AHK10}
 N. Alon, D. Hefetz and M. Krivelevich,
{\em Playing to retain the advantage}, Combin. Probab.
Comput. 19 (2010), 481--491.
\bibitem{AKSS05}
N. Alon, M. Krivelevich, J. Spencer and T. Szab\'o, {\em
Discrepancy games},
Electron. J. Combin.  12 (1) (2005), publ. R51.
\bibitem{AS}
N. Alon and J. Spencer,
{\bf The Probabilistic Method},
3rd ed., Wiley, 2008.
\bibitem{AJSS08}
V. Anuradha, C. Jain, J. Snoeyink and T. Szab\'o,
{\em How long can a graph be kept planar?},
Electron. J. Combin. 15(1) (2008), publ. N14.
\bibitem{BP12}
 J. Balogh and A. Pluh\'ar, {\em The positive minimum degree game on sparse graphs},
 Electron. J. Combin.  19  (2012),  Publ. 22.
\bibitem{BS11}
J. Balogh and W. Samotij,
{\em On the Chv\'atal-Erd\H{o}s triangle game},
 Electron. J. Combin. 18 (2011), publ. 72.
\bibitem{Beck81}
J. Beck, {\em van der Waerden and Ramsey type games},
Combinatorica  1  (1981),  103--116.
\bibitem{Beck82}
J. Beck,
{\em Remarks on positional games},
Acta Math. Acad. Sci. Hungar.  40 (1982), 65--71.
\bibitem{Beck94}
J. Beck,
{\em Deterministic graph games and a probabilistic intuition},
 Combin. Probab. Comput. 3 (1994), 13--26.
\bibitem{Beck02}
J. Beck,
{\em Ramsey games},
Discrete Math. 249 (2002), 3--30.
\bibitem{Beckbook}
J. Beck,
{\bf Combinatorial Games: Tic-Tac-Toe Theory},
Encyclopedia of Mathematics and Its Applications 114, Cambridge University Press, 2008.
\bibitem{Beckbook2}
J. Beck,
{\bf Inevitable randomness in discrete mathematics},
University Lecture Series, 49. Amer. Math. Soc., Providence, RI, 2009.
\bibitem{B13}
M. Bednarska-Bzd\c{e}ga,
{\em On weight function methods in Chooser-Picker games},
Theor. Comput. Sci. 475 (2013), 21--33.
\bibitem{B14}
M. Bednarska-Bzd\c{e}ga,
{\em Avoider-Forcer games on hypergraphs with small rank},
 Electron. J. Combin. 21 (1) (2014), publ. P1.2.
\bibitem{BL00}
M. Bednarska and T. \L uczak,
{\em Biased positional games for which random strategies are nearly optimal},
Combinatorica 20 (2000), 477--488.
\bibitem{BFHK12}
 S. Ben-Shimon, A. Ferber, D. Hefetz and M. Krivelevich,
{\em Hitting time results for Maker-Breaker games},
Random Struct. Alg. 41 (2012), 23--46.
\bibitem{BK13}
T. Bohman and P. Keevash,
{\em Dynamic concentration of the triangle-free process},
submitted, arXiv:1302.5963 [math.CO].
\bibitem{Bol-rgbook}
B. Bollob\'{a}s,
{\bf Random Graphs},
Academic Press, London, 1985.
\bibitem{CE78}
V. Chv\'atal and P. Erd\H os,
{\em Biased positional games},
Annals Discrete Math. 2 (1978), 221--228.
\bibitem{CEPT14}
 D. Clemens, J. Ehrenm\"uller, Y. Person and T. Tran,
{\em Keeping Avoider's graph almost acyclic},
submitted, arXiv:1403.1482 [math.CO].
\bibitem{CGL14}
D. Clemens, H. Gebauer and A. Liebenau,
{\em  The random graph intuition for the tournament game},
submitted, arXiv:1307.4229 [math.CO].
\bibitem{C10}
A. Csernenszky,
{\em The Picker-Chooser diameter game},
Theor. Comput. Sci. 411 (2010), 3757--3762.
\bibitem{CMP09}
A. Csernenszky, C. I. M\'andity and A. Pluh\'ar,
{\em On Chooser-Picker positional games},
Discrete Math. 309 (2009), 5141--5146.
\bibitem{DLR98}
D. Duffus, T \L uczak and  V. R\"odl,
{\em  Biased positional games on hypergraphs}.
 Studia Sci. Math. Hungar.  34  (1998),  141--149.
\bibitem{Erdos61}
P. Erd\H{o}s, {\em Graph theory and probability II},
Can. J. Math., 13 (1961),  346--352.
\bibitem{EL75}
P. Erd\H os and L. Lov\' asz,
{\em Problems and results on 3-chromatic hypergraphs and some related questions}, In:
Infinite and finite sets II, Colloq. Math. Soc. J. Bolyai, Vol. 10, North-Holland, 1975, 609--627.
\bibitem{ES73}
P. Erd\H{o}s and J. L. Selfridge,
{\em On a combinatorial game},
J.  Combin. Th. Ser. A 14 (1973), 298--301.
\bibitem{ESW95}
P. Erd\H{o}s, S. Suen and  P.  Winkler, {\em On the size of a random maximal graph},
Random Struct. Alg.  6  (1995), 309--318.
\bibitem{FK08}
 O. Feldheim and M. Krivelevich,
{\em Winning fast in sparse graph construction games},
 Combin. Probab. Comput.  17 (2008), 781--791.
\bibitem{FGKLPVV13}
 A. Ferber, R. Glebov, M. Krivelevich, H. Liu, C. Palmer, T.
Valla and M. Vizer,
{\em The biased odd cycle game},
 Electron. J. Combin. 20 (3) (2013),
publ. P9.
\bibitem{FGKN14}
 A. Ferber, R. Glebov, M. Krivelevich and A. Naor,
{\em Biased games on random boards},
Random Struct. Alg., to appear.
\bibitem{FH11}
A. Ferber and D. Hefetz,
{\em Winning strong games through fast strategies for weak games},
 Electron. J. Combin. 18(1) (2011), publ.  144.
\bibitem{FH14}
A. Ferber and D. Hefetz,
{\em Weak and strong $k$-connectivity games},
 Europ. J. Combin. 35 (2014), 169--183.
 \bibitem{FKN14}
 A. Ferber, M. Krivelevich and H. Naves,
 {\em Generating random graphs in biased Maker-Breaker games},
 submitted,  arXiv:1310.4096 [math.CO].
\bibitem{FGM13}
G. Fiz Pontiveros, S. Griffiths and R. Morris,
{\em The triangle-free process and $R(3,k)$},
submitted, arXiv:1302.6279 [math.CO].
\bibitem{FKPS05}
 A. Frieze, M. Krivelevich, O. Pikhurko and T. Szab\'o, {\em The
game of JumbleG},
 Combin. Probab. Comput. 14 (2005),
783--793.
\bibitem{Gale79}
D. Gale,
{\em The game of Hex and the Brouwer fixed-point theorem},
 Amer. Math. Monthly 86 (1979), 818--827.
\bibitem{Geb12-2}
H. Gebauer,
{\em Disproof of the Neighborhood Conjecture with Implications to SAT},
Combinatorica 32 (2012), 573--587.
\bibitem{Geb12}
H. Gebauer,
{\em On the Clique-Game},
Europ. J. Combin. 33 (2012), 8--19.
\bibitem{Geb13}
H. Gebauer, {\em Size Ramsey number of bounded degree graphs for games},
Combin. Probab. Comput.  22  (2013),  499--516.
\bibitem{GS09}
H. Gebauer and T. Szab\'o,
{\em Asymptotic random graph intuition for the biased connectivity game},
Random Struct. Alg. 35 (2009), 431--443.
\bibitem{GST11}
H. Gebauer, T. Szab\'o and G. Tardos,
{\em The Local Lemma is tight for SAT},
22nd Annual Symposium on Discrete Algorithms (SODA 2011), 664--674.
\bibitem{HJ63}
A. W. Hales and R. I. Jewett,
{\em Regularity and positional games},
Trans.  Amer.  Math. Soc. 106 (1963), 222--229.
\bibitem{HL}
Y. O. Hamidoune and M. Las Vergnas,
{\em A solution to the box game},
Discrete Math. 65 (1987), 157--171.
\bibitem{HKSS08}
 D. Hefetz, M. Krivelevich, M. Stojakovi\'c and T. Szab\'o,
{\em Planarity, colorability and minor games}, SIAM J.  Discrete
Math. 22 (2008), 194--212.
\bibitem{HKSS09-2}
 D. Hefetz, M. Krivelevich, M. Stojakovi\'c and T. Szab\'o,
{\em Fast winning strategies in Maker-Breaker games},
J. Combin. Th. Ser. B 99 (2009), 39--47.
\bibitem{HKSS10}
 D. Hefetz,  M. Krivelevich, M. Stojakovi\'c and T. Szab\'o,
{\em Avoider -- Enforcer: the rules of the game},
J. Combin. Th. Ser. A 117 (2010), 152--163.
\bibitem{HKSS11}
 D. Hefetz, M. Krivelevich, M. Stojakovi\'c and T. Szab\'o,
{\em Global Maker-Breaker games on sparse graphs},
Europ. J. Combin. 32 (2011), 162--177.
\bibitem{HKSSbook}
 D. Hefetz, M. Krivelevich, M. Stojakovi\'c and T. Szab\'o,
{\bf Positional Games (Oberwolfach Seminars)},
Birkh\"auser, 2014.
\bibitem{HKS07}
 D. Hefetz, M. Krivelevich and T. Szab\'o, {\em
Avoider-Enforcer games},
J. Combin. Th. Ser. A 114 (2007), 840--853.
\bibitem{HKS07-2}
 D. Hefetz, M. Krivelevich and T. Szab\'o, {\em Bart-Moe
games, JumbleG and discrepancy},
Europ. J. Combin. 28 (2007), 1131--1143.
\bibitem{HS09}
D. Hefetz and S. Stich,
{\em On two problems regarding the Hamilton cycle game},
 Electron. J. Combin. 16 (1) (2009), publ. R28.
\bibitem{JLR}
S. Janson, T. \L uczak and  A. Ruci\'nski,
{\bf Random graphs},
Wiley, 2000.
\bibitem{Kim95}
J. H. Kim,
{\em The Ramsey number $R(3,t)$ has order of magnitude $t^2/\log t$},
Random Struct. Alg. 7 (1995), 173--207.
\bibitem{Knox12}
F. Knox,
{\em Two constructions relating to conjectures of Beck on positional games},
manuscript, arXiv:1212.3345 [math.CO].
\bibitem{K95}
 M. Krivelevich,
{\em Bounding Ramsey numbers through large deviation
inequalities},
Random Struct. Alg. 7 (1995), 145--155.
\bibitem{K11}
 M. Krivelevich,
{\em The critical bias for the Hamiltonicity game is $(1+o(1))n/\ln n$},
J. Amer. Math. Soc. 24 (2011), 125--131.
\bibitem{KS06}
 M. Krivelevich and B. Sudakov, {\em Pseudo-random graphs},
In: More sets, graphs and numbers, Bolyai Soc. Math. Stud. Vol. 15, 2006, 199--262
\bibitem{KS08}
 M. Krivelevich and T. Szab\'o,
{\em Biased positional games and small hypergraphs with large covers},
 Electron. J. Combin. 15 (1) (2008), publ. R70.
\bibitem{Leh64}
A. Lehman,
{\em A solution of the Shannon switching game},
J.  SIAM 12 (1964), 687--725.
\bibitem{MR}
R. Motwani and P. Raghavan,
{\bf Randomized Algorithms},
Cambridge University Press, 1995.
\bibitem{MS14}
T. M\"uller and M. Stojakovi\'c,
{\em A threshold for the Maker-Breaker clique game},
Random Struct. Alg., to appear.
\bibitem{NSS14}
 R. Nenadov, A. Steger and M. Stojakovi\'c,
{\em On the threshold for the Maker-Breaker $H$-game},
submitted, arXiv:1401.4384 [Math.CO].
\bibitem{Nilli-HJ}
A. Nilli, {\em Shelah's proof of the Hales-Jewett theorem},
 Mathematics of Ramsey theory,  Algorithms Combin., 5, Springer, Berlin, 1990, 150--151.
\bibitem{RR95}
V. R\"odl and A. Ruci\'nski,
{\em Threshold functions for Ramsey properties},
J. Amer. Math. Soc. 8 (1995), 917--942.
\bibitem{RS00}
V. R\"odl and E. Szemer\'edi,
{\em On size Ramsey numbers of graphs with bounded degree},
Combinatorica 20 (2000), 257--262.
\bibitem{Shelah-HJ}
S. Shelah, {\em Primitive recursive bounds for van der Waerden numbers}.
J. Amer. Math. Soc.  1  (1988), 683--697.
\bibitem{S77}
J. Spencer, {\em Asymptotic lower bounds for Ramsey functions},
Discrete Math.  20  (1977),  69--76.
\bibitem{SV08}
B. Sudakov and V. H. Vu,
{\em Local resilience of graphs}, Random Struct. Alg.
33 (2008), 409--433.
\bibitem{SS05}
M. Stojakovi\'c and T. Szab\' o,
{\em Positional games on random graphs},
Random Struct. Alg. 26 (2005), 204--223.
\bibitem{Sze84}
L. A. Sz\'ekely,
{\em On two concepts of discrepancy in a class of combinatorial games}, In:
Finite and Infinite Sets, Colloq. Math. Soc. J. Bolyai, Vol. 37, North-Holland, 1984, 679--683.
\end{thebibliography}
\end{document}